\documentclass[11pt,letterpaper]{article}
\usepackage[margin=1.6in]{geometry}
\usepackage[utf8]{inputenc}
\usepackage{amsmath,amsthm,amsfonts}
\usepackage[mathscr]{euscript}
\usepackage{tikz-cd}
\usepackage{xcolor}
\usepackage{soul,enumerate}

\newtheorem{theorem}{Theorem}[section]

\newtheorem{proposition}[theorem]{Proposition}

\theoremstyle{definition}
\newtheorem{definition}[theorem]{Definition}
\newtheorem{example}[theorem]{Example}
\newtheorem{remark}[theorem]{Remark}
\newtheorem{exercise}[theorem]{Exercise}

\DeclareMathOperator{\Hom}{Hom}
\DeclareMathOperator{\id}{id}
\DeclareMathOperator{\Obj}{Obj}
\DeclareMathOperator{\Img}{Im}
\DeclareMathOperator{\Coim}{Coim}
\DeclareMathOperator{\Ker}{Ker}
\DeclareMathOperator{\cok}{cok}
\DeclareMathOperator{\img}{im}
\DeclareMathOperator{\coim}{coim}
\DeclareMathOperator{\Cok}{Cok}

\newcommand{\C}{\mathbf{C}}
\newcommand{\A}{\mathbf{A}}
\newcommand{\Cop}{\C^{op}}
\newcommand{\Aop}{\A^{op}}
\newcommand{\Set}{\mathbf{Set}}
\newcommand{\Z}{\mathbb{Z}}
\newcommand{\Q}{\mathbb{Q}}
\newcommand{\Grp}{\mathbf{Grp}}
\newcommand{\Ab}{\mathbf{Ab}}
\newcommand{\Ring}{\mathbf{Ring}}
\mathchardef\mhyphen="2D
\newcommand{\RMod}{R\mhyphen\mathbf{Mod}}
\newcommand\blfootnote[1]{%
  \begingroup
  \renewcommand\thefootnote{}\footnote{#1}%
  \addtocounter{footnote}{-1}%
  \endgroup
}

\title{\large{\textbf{SOME NOTES ON DIAGRAM CHASING AND DIAGRAMMATIC PROOFS IN CATEGORY THEORY}}}
\author{\normalsize{VALENTINO VITO}}
\date{}

\begin{document}
\maketitle

\begin{abstract}
\noindent {\footnotesize Diagram chasing is a customary proof method used in category theory and homological algebra. It involves an element-theoretic approach to show that certain properties hold for a commutative diagram. When dealing with abelian categories for the first time, one would work using a diagrammatic approach without relying on the notion of elements. However, constantly manipulating universal properties of various diagrams can be quite cumbersome. That said, we believe that it is still important to draw a contrast between both viewpoints in order to motivate the field of category theory. We focus our scope to the short five lemma, one of the more elementary diagram lemmas, and present a quick exposition on relevant subjects. Moreover, we give an original proof of the short five lemma using the universal property of pullbacks.\\[-2mm]

}
\end{abstract}

\blfootnote{Department of Mathematics, Faculty of Mathematics and Natural Sciences (FMIPA), Universitas Indonesia, Depok 16424, Indonesia\\
Email address: \textbf{valentino.vito@sci.ui.ac.id}}
{\footnotesize\tableofcontents}

\section{Introduction}

Category theory is a prevalent branch of mathematics able to translate various properties of mathematical structures into arrow diagrams. Arrows in this context refer to morphisms, which in some sense can be considered as mappings or arrows between objects. But since our categorical objects are not necessarily in the form of sets, some care is required to be able to communicate the subject without working with elements. A clear advantage of using the language of categories is the gained insight of the big picture and how seemingly different concepts can be unified.

Of course, it would be harder to appreciate the depth and utility of categories without previous exposure to a lot of mathematical structures. We thus assume that the reader is comfortable with the basics of modern algebra. We do provide a short rundown on homological algebra in the context of abelian groups. In particular, we introduce the notion of a diagram chase as a way to motivate the categorical nature of the sections that follow. In addition, to prevent the exposition from becoming too long-winded, we focus the problem to the short five lemma.

The short five lemma is a member of a family of lemmas known as diagram lemmas. As the name might suggest, these lemmas involve deducing some properties of a commutative diagram. The short five lemma is a special case of the more popular five lemma. It is picked as our main subject of interest as it is one of the easier diagram lemmas to state and prove (since most diagram lemmas require a relatively long proof when written down). The term "short" refers to the fact that the lemma deals with short exact sequences, while the term "five" refers to the number of objects in each short exact sequence.

We consider both element-theoretic (via a diagram chase) and arrow-theoretic arguments (by manipulating universal properties in the form of diagrams) to prove the short five lemma. We do not wish to apply advanced results such as the Freyd-Mitchell embedding theorem, which gives the existence of a property-preserving embedding of a small abelian category to some category of $R$-modules.

\begin{theorem}[Freyd-Mitchell embedding theorem, \cite{junhan}]
Given a small abelian category $\A$, there exists a ring $R$ and an exact full embedding $\A \to \RMod$.
\end{theorem}

The Freyd-Mitchell embedding theorem gives us a shortcut to deduce various diagram lemmas for (small) abelian categories from their counterpart in the category of modules. While this theorem is rightly important in the field, we do not require it throughout our discussion.

This article is largely expository and is aimed toward advanced undergraduate and graduate students as an introduction to category theory. We first provide a diagram chase for the short five lemma for abelian groups as to motivate the subject. We then give a quick review of elementary category theory before delving into abelian categories. Afterwards, we build the basic theory of homological algebra to prove the short five lemma for abelian categories. Finally, we present an original proof of the short five lemma in order to give another point of view on the problem. For brevity, we will omit proofs that do not enhance the exposition and refer the reader to other sources in this case. Alternatively, the reader can try to independently fill in these gaps, most of which are listed as an exercise in Section \ref{exer}.

\section{Diagram chasing}\label{chase}

Let $A$, $B$ and $C$ be abelian groups, and let $f\colon A \to B$ and $g\colon B \to C$ be group homomorphisms. We can construct the following sequence of abelian groups:
\begin{equation}\label{seq}\begin{tikzcd}
0 \arrow[r] & A \arrow[r, "f"] & B \arrow[r,"g"] & C \arrow[r] & 0
\end{tikzcd}\end{equation}
where $0$ denotes the trivial group and the homomorphisms $0 \to A$ and $C \to 0$ are clear.

\begin{definition}\label{ext1}
We say that (\ref{seq}) is \textit{exact} if $f$ is injective, $g$ is surjective, and $\Img f = \Ker g$. We then call the sequence a \textit{short exact sequence}.
\end{definition}

In a similar way, we can also define short exact sequences for $R$-modules and groups in general. Short exact sequences form a subfamily of the more general \textit{exact sequences}, which allow the sequence to be (doubly) infinite.

As an illustration of a diagram chase, we prove the short five lemma for abelian groups.

\begin{theorem}[Short five lemma]\label{sfl1}
Let
\begin{equation}\label{seq1}\begin{tikzcd}
0 \arrow[r] & A \arrow[r, "f"] & B \arrow[r,"g"] & C \arrow[r] & 0
\end{tikzcd}\end{equation}
and
\begin{equation}\label{seq2}\begin{tikzcd}
0 \arrow[r] & A' \arrow[r, "f'"] & B' \arrow[r,"g'"] & C' \arrow[r] & 0
\end{tikzcd}\end{equation}
be short exact sequences of abelian groups which are a part of the following diagram:
\begin{equation}\label{seq3}\begin{tikzcd}
0 \arrow[r] & A \arrow[r, "f"] \arrow[d, "\alpha"] & B \arrow[r,"g"] \arrow[d, "\beta"] & C \arrow[r] \arrow[d, "\gamma"] & 0 \\
0 \arrow[r] & A' \arrow[r, "f'"] & B' \arrow[r,"g'"] & C' \arrow[r] & 0
\end{tikzcd}\end{equation}
Assume that the diagram commutes, so that $f'\alpha=\beta f$ and $g'\beta=\gamma g$. We have the following:
\begin{enumerate}[{\normalfont(i)}]
    \item If $\alpha$ and $\gamma$ are injective, then $\beta$ is injective;
    \item If $\alpha$ and $\gamma$ are surjective, then $\beta$ is surjective;
    \item If $\alpha$ and $\gamma$ are bijective, then $\beta$ is bijective.
\end{enumerate}
\end{theorem}

The idea of the proof is to traverse the diagram in (\ref{seq3}), taking elements from each object visited and keeping in mind the commutativity and exactness of the diagram. This method is what is known as "chasing the diagram".

\begin{proof}[Proof of Theorem \ref{sfl1}]
(i) Fix arbitrary elements $b_1,b_2 \in B$ and suppose that $\beta(b_1)=\beta(b_2)$. We will show that $b_1=b_2$. By commutativity, we have
\[\gamma g(b_1)=g'\beta(b_1)=g'\beta(b_2)=\gamma g(b_2).\]
Since $\gamma$ is injective by assumption, it follows that $g(b_1)=g(b_2)$, or alternatively, $g(b_1-b_2)=0$. This implies that $b_1-b_2 \in \Ker g=\Img f$ by exactness of (\ref{seq1}). Hence, there exists $a \in A$ such that $f(a)=b_1-b_2$. We once again use the fact that the diagram commutes to obtain
\[f'\alpha(a)=\beta f(a)=\beta(b_1-b_2).\]
Since $\beta(b_1-b_2)=0$ and both $f'$ and $\alpha$ are injective, we have $a=0$. This shows that $b_1-b_2=f(a)=0$, which completes our proof.\\ \\
(ii) Fix an arbitrary element $b \in B'$. We want to show that $\beta(z)=b$ for some $z \in B$. Since $\gamma$ and $g$ are surjective, there is an $x \in B$ such that $\gamma g(x)=g'(b)$. Hence, by commutativity,
\[g'\beta(x)=\gamma g(x)=g'(b),\]
which gives us $g'(\beta(x)-b)=0$. By exactness of (\ref{seq2}), it follows that $\beta(x)-b \in \Ker g' = \Img f'$. Recalling that $\alpha$ is surjective by hypothesis, there is an $a \in A$ such that $f'\alpha(a)=\beta(x)-b$. Using commutativity once more, we obtain
\[\beta f(a)=f'\alpha(a)=\beta(x)-b.\]
Thus we have $\beta(x-f(a))=b$, as desired.\\ \\
(iii) This obviously follows from (i) and (ii).
\end{proof}

\begin{remark}
The short five lemma still holds for groups in general. For the sake of continuity for future sections, we limit our scope to abelian groups.
\end{remark}

As we can see, it requires a lot of notational juggling not only to prove the lemma, but also to state it. This is one reason why many people do not like to write down every detail of a diagram chase. Instead, we would usually present the explanation visually by sequentially pointing at which object is currently under consideration and mentioning how to move from one object of the diagram to another object. For example, in explaining the chase for Theorem \ref{sfl1}(ii), one would first point at $b \in B'$ before moving on to $g'(b) \in C'$. Then one would claim surjectivity and commutativity to pick some element $x \in B$ and consider $\beta(x)-b \in B'$, from which exactness is applied so that we can circle around to grab $a \in A$, and so on.

An interesting facet of the proof is that it does not explicitly mention that we are working in abelian groups. In fact, if we were to prove the short five lemma for $R$-modules, we could copy the previous arguments verbatim to obtain the proof. This leads us to ask whether we can create a more general setting where the concepts of exactness and commutativity are defined so that these ideas are still preserved in this new setting. This is where the structure of abelian categories comes into play.

The next few sections are dedicated to quickly build the theory of abelian categories. We will not delve into the subtleties of modern homological algebra. Nonetheless, we will present some insights on how abelian categories can serve as a very abstract, yet intriguing construction to generalize the structure of abelian groups. In Section \ref{gsfl}, for example, we show how Theorem \ref{sfl1}(i) and (ii) are "dual" statements in some sense.

\section{Some categorical background}

Here, we give a rapid treatment of the category theory needed to introduce abelian categories. Readers unfamiliar with categories can consult \cite{lang,leinster} for a more comprehensive introduction.

\begin{definition}\label{cats}
A \textit{category} $\C$ consists of:
\begin{enumerate}
    \item A class of \textit{objects} denoted as $\Obj(\C)$;
    \item For every $A,B \in \Obj(\C)$, a set $\Hom(A,B)$ called a \textit{hom-set} consisting of \textit{morphisms} (where $A$ is the \textit{source} and $B$ is the \textit{target}) which can be written as $f\colon A \to B$;
    \item For every $A,B,C \in \Obj(\C)$, a composition rule
    \[\circ\colon \Hom(B,C) \times \Hom(A,B) \to \Hom(A,C)\]
    which sends $g \in \Hom(B,C)$ and $f \in \Hom(A,B)$ to $gf=g \circ f \in \Hom(A,C)$,
\end{enumerate}
such that the following properties hold:
\begin{enumerate}
    \item The hom-sets $\Hom(A,B)$ and $\Hom(C,D)$ are disjoint unless $A=C$ and $B=D$;
	\item For every $A \in \Obj(\C)$, there exists an \textit{identity morphism} $\id_A \in \Hom(A,A)$ such that $f \id_A=f$ and $\id_A g=g$ for all appropriate morphisms $f$ and $g$ (appropriate here means that it makes sense to write $f \id_A$ and $\id_A g$);
	\item The composition rule is associative, that is, $(hg)f=h(gf)$ for all appropriate morphisms $f,g,h$.
\end{enumerate}
\end{definition}

Informally, a class can be considered as a collection much like a set but is allowed to contain other mathematical objects without the fear that it might be too big. For example, unlike a set which cannot contain all sets, we can construct a class of all sets without running into foundational issues. This informal definition of a class is good enough for our purposes.

\begin{remark}
Our definition requires hom-sets to be sets. Some authors allow hom-sets to be proper classes, while defining \textit{locally small} categories to coincide with our definition of categories.
\end{remark}

\begin{example}
As a prototypical example, $\Set$ is a category such that $\Obj(\Set)$ is the class of all sets, each $\Hom(A,B)$ is the set of all set-functions $A \to B$, and the composition rule is precisely the standard set-function composition. The axioms listed in Definition \ref{cats} can be shown to hold for $\C=\Set$.
\end{example}

\begin{example}
Other examples include the category of all groups and group homomorphisms denoted as $\Grp$ and its \textit{full} subcategory $\Ab$ of abelian groups. By full, we mean that $\Hom_\Ab(A,B)=\Hom_\Grp(A,B)$ for every pair of abelian groups $A$ and $B$. Given a ring $R$, the category of $R$-modules and module homomorphisms is denoted as $\RMod$.
\end{example}

\begin{example}
Given a category $\C$, its \textit{opposite} (or \textit{dual}) $\Cop$ is a category such that $\Obj(\Cop)=\Obj(\C)$ and $\Hom_{\Cop}(A,B)=\Hom_\C(B,A)$ for all objects $A$ and $B$. Additionally, the composition $\circ_{\Cop}$ in $\Cop$ is obtained from $\circ_\C$ by the identity $f \circ_{\Cop} g=g \circ_\C f$. We have that $(\Cop)^{op}=\C$.
\end{example}

An opposite category is, roughly speaking, obtained from the original category by "reversing" its arrows/morphisms. Diagram commutativity is preserved after reversing arrows to the opposite category. For illustration, a diagram in an ambient category $\C$:
\begin{equation}\label{tri1}\begin{tikzcd}
A \arrow[d,"f"'] \arrow[r, "h"] & B \\
C \arrow[ur,"g"']
\end{tikzcd}\end{equation}
becomes a diagram in an ambient category $\Cop$ after the arrows are reversed as such:
\begin{equation}\label{tri2}\begin{tikzcd}
A & B \arrow[l, "h"'] \arrow[dl,"g"] \\
C \arrow[u,"f"]
\end{tikzcd}\end{equation}
Observe that (\ref{tri1}) commutes if and only if (\ref{tri2}) commutes---that is to say, $f \circ_\C g=h$ if and only if $g \circ_{\Cop} f=h$.

\begin{definition}
Let $f\colon A \to B$ be a morphism. It is a \textit{monomorphism} (or is \textit{monic}) if for every pair of morphisms $\alpha',\alpha''$ such that $f\alpha'=f\alpha$, we have $\alpha'=\alpha''$.
It is an \textit{epimorphism} (or is \textit{epic}) if for every pair of morphisms $\beta',\beta''$ such that $\beta' f=\beta'' f$, we have $\beta'=\beta''$. It is an \textit{isomorphism} if it has an inverse $g\colon B \to A$ such that $gf=\id_A$ and $fg=\id_B$. Two objects $A$ and $B$ are \textit{isomorphic} if there exists an isomorphism $f\colon A \to B$.
\end{definition}

The inverse of an isomorphism $f$ is necessarily unique, so we denote it as $f^{-1}$. Monomorphisms and epimorphisms are \textit{dual} in the sense that monomorphisms in $\C$ are precisely epimorphisms in $\Cop$. An isomorphism is both a monomorphism and an epimorphism, but the converse need not hold. For example, in the category $\Ring$ of rings, the inclusion map $\iota\colon \Z \to \Q$ can be proved to be both monic and epic even though it does not have an inverse.

In certain categories such as $\Set$ and $\Ab$, monomorphisms, epimorphisms and isomorphisms are precisely injective, surjective and bijective functions, respectively. In fact, we can consider the aforementioned three properties of morphisms to be a spiritual representation of injection, surjection and bijection. These equivalences do not translate well to all categories, however, as the inclusion $\iota\colon \Z \to \Q$ is epic in $\Ring$ but not surjective.

\begin{definition}
An object $I \in \Obj(\C)$ is called \textit{initial} if it admits a unique morphism $I \to A$ for every $A \in \Obj(\C)$. An object $T \in \Obj(\C)$ is called \textit{terminal} if it admits a unique morphism $A \to T$ for every $A \in \Obj(\C)$. A \textit{zero object} $0 \in \Obj(\C)$ is both initial and terminal.
\end{definition}

Initial, terminal and zero objects are each unique up to unique isomorphism whenever they exist, so we usually say \textit{the} initial/terminal/zero object. This abuse of language is commonplace, and we will continue this practice with other categorical properties such as products and kernels in the future. As with monomorphisms and epimorphisms, initial and terminal objects are dual concepts since an initial object in $\C$ is terminal in $\Cop$, and vice versa.

\begin{example}
The empty set $\emptyset$ is initial in $\Set$, with the empty function $\emptyset \to A$ being the unique morphism to any object $A$. On the other hand, any singleton set $\{a\}$ is terminal in $\Set$, with the constant function being the unique morphism to the singleton. Since the initial and terminal objects are not isomorphic, $\Set$ does not contain a zero object. In contrast, $\Grp$ and $\Ab$ both have the trivial group as a zero object.
\end{example}

We can consider diagrams that are initial/terminal with respect to other diagrams. Fix $A,B \in \Obj(\C)$. Suppose that
\begin{equation}\label{prod}\begin{tikzcd}[row sep=14 pt]
&A\\
C \arrow[ur,"\varphi"] \arrow[dr,"\psi"'] \\
&B
\end{tikzcd}\end{equation}
is such that for every other diagram
\[\begin{tikzcd}[row sep=14 pt]
&A\\
C' \arrow[ur,"\varphi'"] \arrow[dr,"\psi'"'] \\
&B
\end{tikzcd}\]
we have a unique $h\colon C' \to C$ such that the following diagram commutes:
\[\begin{tikzcd}[row sep=14 pt]
&&A \\
C' \arrow[urr,bend left,"\varphi'"] \arrow[r,dashrightarrow,"\exists! h"] \arrow[drr,bend right,"\psi'"'] &C \arrow[ur,"\varphi"] \arrow[dr,"\psi"'] \\
&&B
\end{tikzcd}\]
We can say that the diagram in (\ref{prod}) is terminal in the category of diagrams of its type, and in this particular case we usually refer to this diagram as a product. We can define the dual notion of coproducts in a similar way via initial diagrams instead.

(Co)products are both examples of a standard construction in category theory known as a universal property. Simply put, a diagram is said to satisfy a \textit{universal property} if it is universal (initial/terminal) with respect to other diagrams of the same type.

\begin{definition}\label{(co)pro}
Fix two objects $A$ and $B$ of a category $\C$. The \textit{product} of $A$ and $B$ is a triple $(A \times B,\rho_A,\rho_B)$, where $A \times B \in \Obj(\C)$, $\rho_A\colon A \times B \to A$ and $\rho_B\colon A \times B \to B$, such that for every triple $(C,\varphi,\psi)$ of the same type, there exists a unique morphism $h\colon C \to A \times B$ making the following diagram commute:
\begin{equation}\label{pro}\begin{tikzcd}[row sep=14 pt]
&&A \\
C \arrow[urr,bend left,"\varphi"] \arrow[r,dashrightarrow,"\exists! h"] \arrow[drr,bend right,"\psi"'] &A \times B \arrow[ur,"\rho_A"] \arrow[dr,"\rho_B"'] \\
&&B
\end{tikzcd}\end{equation}
The \textit{coproduct} of $A$ and $B$ is a triple $(A \amalg B,\iota_A,\iota_B)$, where $A \amalg B \in \Obj(\C)$, $\iota_A\colon A \to A \amalg B$ and $\iota_B\colon B \to A \amalg B$, such that for every triple $(C,\varphi,\psi)$ of the same type, there exists a unique morphism $h\colon A \amalg B \to C$ making the following diagram commute:
\begin{equation}\label{copro}\begin{tikzcd}[row sep=14 pt]
&&A \arrow[dll,bend right,"\varphi"']\arrow[dl,"\iota_A"'] \\
C &A \amalg B \arrow[l,dashrightarrow,"\exists! h"'] \\
&&B \arrow[ull,bend left,"\psi"] \arrow[ul,"\iota_B"]
\end{tikzcd}\end{equation}
\end{definition}

Products and coproducts are dual concepts, as seen in how (\ref{pro}) and (\ref{copro}) are both the "reversal" of the other. Like other universal properties we will encounter later, (co)products do not necessarily exist, but they are unique up to unique isomorphism. That is, given two products (resp. coproducts) $(C,\varphi,\psi)$ and $(C',\varphi',\psi')$, there exists a unique isomorphism $f\colon C \to C'$ such that $\varphi=\varphi'f$ (resp. $\varphi'=f\varphi$) and $\psi=\psi'f$ (resp. $\psi'=f\psi$). This justifies us using the term \textit{the} (co)product.

\begin{example}
Given sets $A$ and $B$, the Cartesian product $A \times B$ along with its projection maps $A \times B \to A$ and $A \times B \to B$ form the product of $A$ and $B$ in $\Set$. On the other hand, the disjoint union $A \amalg B$ along with its natural inclusion maps $A \to A \amalg B$ and $B \to A \amalg B$ form the coproduct of $A$ and $B$. In $\Ab$, the direct sum $A \oplus B$ of two abelian groups $A$ and $B$ is both the product and coproduct of $A$ and $B$ (when equipped with natural projections and inclusions, respectively).
\end{example}

\section{Abelian categories}

An abelian category is a category with some added structure on each of its hom-sets and satisfies a set of axioms. One can look at $\Ab$ as the prototypical example of abelian categories. We first provide the definition of additive categories as a stepping stone to define kernels and their dual, which are in turn used to introduce abelian categories.

\begin{definition}
A category $\A$ is \textit{additive} if:
\begin{enumerate}
		\item Every hom-set $\Hom(A,B)$ is equipped with a binary operation $+=+_{AB}$ such that it becomes an abelian group whose identity $0=0_{AB}$ is called the \textit{zero morphism};
		\item The distributive laws hold, that is,
        \[\alpha(f+g)=\alpha f+\alpha g \quad \text{and} \quad (f+g)\beta=f\beta+g\beta\]
        for all appropriate morphisms $\alpha,\beta,f,g$;
		\item $\A$ contains a zero object $0$ along with all products and coproducts (i.e., $A \times B$ and $A \amalg B$ exist for all $A,B \in \Obj(\A)$).
	\end{enumerate}
\end{definition}

As is standard, we drop the index of $0_{AB}$ and simply write $0$, which should not be confused with the zero object. It easily follows from the distributive laws that $f0=0$ and $0g=0$ for all morphisms $f$ and $g$. It is important to note that $\Aop$ is also additive if $\A$ is, with $+_{\Aop}$ defined the same way as $+_\A$. Hence, we say that additive categories are \textit{self-dual}. This property allows the dual of a statement that is true in the theory of additive categories to also be true in the theory.

\begin{proposition}\label{add-mono-epi}
Let $\A$ be an additive category and $f\colon A\to B$ a morphism. Then:
\begin{enumerate}[{\normalfont(i)}]
		\item $f$ is monic if and only if for every morphism $\alpha$ such that $f\alpha=0$, we have $\alpha=0$;
		\item $f$ is epic if and only if for every morphism $\beta$ such that $\beta f=0$, we have $\beta=0$.
\end{enumerate}
\end{proposition}

\begin{proof}
(i) Suppose that $f$ is monic and that $f\alpha=0$. Since $f\alpha=0=f0$, we obtain $\alpha=0$ by left-canceling $f$. Conversely, suppose that $f\alpha=0$ implies $\alpha=0$. If $\alpha',\alpha''$ are such that $f\alpha'=f\alpha''$, we clearly have $f(\alpha'-\alpha'')=0$, thus $\alpha'-\alpha''=0$ by hypothesis. Hence, $\alpha'=\alpha''$, which proves that $f$ is monic.\\ \\
(ii) This is the dual statement of (i) and can be similarly proved by mirroring the previous proof. Alternatively, the proposition follows from the following string of equivalences:
\begin{align*}
\text{$f$ is epic in $\A$} &\iff \text{$f$ is monic in $\Aop$}\\
&\iff \text{$f\circ_{\Aop} \alpha=0$ implies $\alpha=0$}\\
&\iff \text{$\beta\circ_\A f=0$ implies $\beta=0$}.
\end{align*}
The preceding argument switches the focus around between $\A$ and $\Aop$ in order to prove (ii) from (i).
\end{proof}

Proposition \ref{add-mono-epi} presents a pair of \textit{dual statements}, which are similar-looking propositions where (i) talks about a certain categorical concept and (ii) talks about its dual concept. Proposition \ref{add-mono-epi}(ii) can be seen as a "reversed" version of (i) which talks about epimorphisms instead of monomorphisms. A proof of a dual statement can be mirrored from the proof of the original statement or obtained by a clever use of the opposite category. We will encounter more dual statements in this section after we define abelian categories.

\begin{definition}
Let $f\colon A \to B$ be a morphism of an additive category $\A$. The \textit{kernel} of $f$ is the morphism $\ker f\colon K_f \to A$ such that $f(\ker f)=0$, and for every morphism $g\colon C \to A$ such that $fg=0$, there exists a unique $h\colon C \to K_f$ making the following diagram commute:
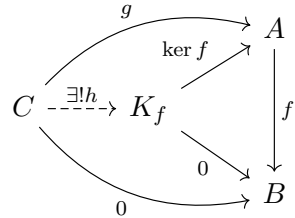
\[\begin{tikzcd}[row sep=14 pt]
&&A \arrow[dd,"f"] \\
C \arrow[urr,bend left,"g"] \arrow[r,dashrightarrow,"\exists! h"] \arrow[drr,bend right,"0"'] &K_f \arrow[ur,"\ker f"] \arrow[dr,"0"'] \\
&&B
\end{tikzcd}\]
Dually, the \textit{cokernel} of $f$ is the morphism $\cok f\colon B \to C_f$ such that $(\cok f)f=0$, and for every morphism $g\colon B \to C$ such that $gf=0$, there exists a unique $h\colon C_f \to C$ making the following diagram commute:
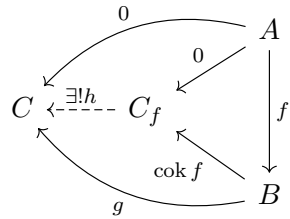
\[\begin{tikzcd}[row sep=14 pt]
&&A \arrow[dd,"f"]\arrow[dll,bend right,"0"']\arrow[dl,"0"'] \\
C &C_f \arrow[l,dashrightarrow,"\exists! h"'] \\
&&B \arrow[ull,bend left,"g"] \arrow[ul,"\cok f"]
\end{tikzcd}\]
\end{definition}

\begin{remark}
We write (co)kernels in the sense of category theory in all lowercase letters (i.e., $\ker$ and $\cok$) instead of writing the first letter in uppercase, which we do when we intend for them to be a set (e.g. $\Ker f=\{x \in A:f(x)=0\}$).
\end{remark}

\begin{definition}\label{abelian}
An additive category $\A$ is \textit{abelian} if:
\begin{enumerate}
		\item $\A$ contains all kernels and cokernels (i.e., $\ker f$ and $\cok f$ exist for every morphism $f$);
		\item Every monomorphism is a kernel of some morphism, and every epimorphism is a cokernel of some morphism.
\end{enumerate}
\end{definition}

Abelian categories provide a nice setting to work with since we can always assume the existence of a zero object, (co)products and (co)kernels. Furthermore, the second axiom of Definition \ref{abelian} is motivated by the fact that subgroups of an abelian group are normal. Like additive categories, abelian categories are self-dual since the opposite of an abelian category is also abelian.

\begin{example}
By defining the sum of two homomorphisms $f,g \in \Hom(A,B)$ as $(f+g)(a)=f(a)+g(a)$, it can be shown that $\Ab$ (and $\RMod$ in general) are abelian. Given a morphism $f\colon A \to B$, the morphism which satisfies the universal property of kernels is the inclusion map $\Ker f \to A$. On the other hand, given $\Cok f=B/\Img f$, the cokernel of $f$ in the sense of category theory is the natural map $B \to \Cok f$ sending each element of $B$ to its corresponding equivalence class.
\end{example}

To provide examples of diagrammatic arguments, we prove a couple of elementary results around (co)kernels. Note that while the following propositions are almost trivial for $\Ab$, we need to rely on diagrams obtained from relevant universal properties when working in the more general setting.

\begin{proposition}
Let $f\colon A\to B$ be a morphism of an abelian category. Then $\ker f$ is monic and $\cok f$ is epic.
\end{proposition}
\begin{proof}
We only prove that $\ker f$ is monic since the other statement easily follows by duality. Suppose $\alpha\colon C \to K_f$ is such that $(\ker f)\alpha=0$. We want to show that $\alpha=0$. Since $f(\ker f)\alpha=0\alpha=0$, we have the following commutative diagram by invoking the universal property of the kernel:
\[\begin{tikzcd}[row sep=14 pt]
&&A \arrow[dd,"f"] \\
C \arrow[urr,bend left,"(\ker f)\alpha=0"] \arrow[r,dashrightarrow,"\exists! h"] \arrow[drr,bend right,"0"'] &K_f \arrow[ur,"\ker f"] \arrow[dr,"0"'] \\
&&B
\end{tikzcd}\]
Observe that $h=0$ and $h=\alpha$ both make the diagram commute. Hence, by the uniqueness of $h$, we have $\alpha=0$.
\end{proof}

\begin{proposition}\label{kereq0}
Let $f\colon A\to B$ be a morphism of an abelian category. Then:
\begin{enumerate}[{\normalfont(i)}]
		\item $f$ is monic if and only if $\ker f=0$.
		\item $f$ is epic if and only if $\cok f=0$.
\end{enumerate}
\end{proposition}
\begin{proof}
We only prove (i). If $f$ is monic, then $\ker f=0$ is obtained by left-canceling $f$ from the equation $f(\ker f)=0$. Now suppose that $\ker f=0$ and $\alpha\colon C \to A$ is such that $f\alpha=0$. We want to prove that $\alpha=0$. The following diagram commutes:
\[\begin{tikzcd}[row sep=14 pt]
&&A \arrow[dd,"f"] \\
C \arrow[urr,bend left,"\alpha"] \arrow[r,dashrightarrow,"\exists! h"] \arrow[drr,bend right,"0"'] &K_f \arrow[ur,"\ker f"] \arrow[dr,"0"'] \\
&&B
\end{tikzcd}\]
We thus have $\alpha=(\ker f)h=0h=0$, and so $f$ is monic.
\end{proof}

We further state a few important properties of abelian categories without proof. While the proofs are available in \cite{aluffi}, one would gain more experience by proving these as an exercise in Section \ref{exer}. 

\begin{proposition}\label{kercok}{\normalfont\cite[p. 565]{aluffi}}
Let $f\colon A \to B$ be a morphism in an abelian category. If $f$ is monic, then $f$ is the kernel of $\cok f$.
\end{proposition}

\begin{proposition}\label{isomonoepi}{\normalfont\cite[p. 566]{aluffi}}
A morphism $f$ of an abelian category that is both a monomorphism and an epimorphism is an isomorphism.
\end{proposition}

\begin{proposition}\label{sum}{\normalfont\cite[p. 568]{aluffi}}
In an abelian category $\A$, products and coproducts coincide. That is, $A \times B \cong A \amalg B$ for all objects $A,B \in \Obj(\A)$.
\end{proposition}

Following Proposition \ref{sum}, we can refer to both products and coproducts in abelian categories as the \textit{direct sum} $A \oplus B$.

\section{The generalized short five lemma}\label{gsfl}

To state the short five lemma for abelian categories, we first provide a more general definition for short exact sequences.

\begin{definition}\label{ext2}
We say that the following sequence in an abelian category $\A$:
\begin{equation}\label{seq4}\begin{tikzcd}
0 \arrow[r] & A \arrow[r, "f"] & B \arrow[r,"g"] & C \arrow[r] & 0
\end{tikzcd}\end{equation}
is \textit{exact} if $f$ is monic, $g$ is epic, $gf=0$ and $(\cok f)(\ker g)=0$. We then call the sequence a \textit{short exact sequence}.
\end{definition}

\begin{remark}
The symbol $0$ in (\ref{seq4}) represents the zero object in $\A$, which renders the morphisms $0 \to A$ and $C \to 0$ unambiguous. This coincides with our previous notation of $0$ denoting the trivial group, as it is indeed the zero object in $\Ab$. In fact, the conditions $gf=0$ and $(\cok f)(\ker g)=0$ in Definition \ref{ext2} precisely recovers the exactness property $\Img f=\Ker g$ in Definition \ref{ext1} for $\A=\Ab$.
\end{remark}

With some thought, Definition \ref{ext2} can be easily generalized to define exact sequences that are not necessarily short. But for short exact sequences in particular, we prefer a more compact definition as follows:

\begin{proposition}\label{ext3}
The sequence in {\normalfont(\ref{seq4})} is exact if and only if $f=\ker g$ and $g=\cok f$.
\end{proposition}
\begin{proof}
If $f=\ker g$ and $g=\cok f$, it is easy to see that the four conditions in Definition \ref{ext2} all hold. Conversely, we just need to show that $f=\ker g$, for $g=\cok f$ follows by duality. From $gf=0$, we have a commutative diagram as follows:
\[\begin{tikzcd}[row sep=14 pt]
&&B \arrow[dd,"g"] \\
A \arrow[urr,bend left,"f"] \arrow[r,dashrightarrow,"\exists! \varphi"] \arrow[drr,bend right,"0"'] &K_g \arrow[ur,"\ker g"] \arrow[dr,"0"'] \\
&&C
\end{tikzcd}\]
In addition, since $(\cok f)(\ker g)=0$ and $f$ is monic, an application of Proposition \ref{kercok} yields the following commutative diagram:
\[\begin{tikzcd}[row sep=14 pt]
&&B \arrow[dd,"\cok f"] \\
K_g \arrow[urr,bend left,"\ker g"] \arrow[r,dashrightarrow,"\exists! \psi"] \arrow[drr,bend right,"0"'] &A \arrow[ur,"f"] \arrow[dr,"0"'] \\
&&C_f
\end{tikzcd}\]
Both $\varphi$ and $\psi$ are easily shown to be inverses of each other. Hence, $\varphi$ is an isomorphism. This proves that $f=\ker g$.
\end{proof}

We are now ready to state and prove the short five lemma. The following proof is due to Mac Lane \cite{maclane-categories,maclane-homology}.

\begin{theorem}\label{sfl2}
Let
\begin{equation}\label{seq5}\begin{tikzcd}
0 \arrow[r] & A \arrow[r, "f"] \arrow[d, "\alpha"] & B \arrow[r,"g"] \arrow[d, "\beta"] & C \arrow[r] \arrow[d, "\gamma"] & 0 \\
0 \arrow[r] & A' \arrow[r, "f'"] & B' \arrow[r,"g'"] & C' \arrow[r] & 0
\end{tikzcd}\end{equation}
be a commutative diagram with exact rows in an abelian category $\A$. We have the following:
\begin{enumerate}[{\normalfont(i)}]
    \item If $\alpha$ and $\gamma$ are monic, then $\beta$ is monic;
    \item If $\alpha$ and $\gamma$ are epic, then $\beta$ is epic;
    \item If $\alpha$ and $\gamma$ are isomorphisms, then $\beta$ is an isomorphism.
\end{enumerate}
\end{theorem}
\begin{proof}
(i) Suppose that $\alpha$ and $\gamma$ are monic. Observe that
\[\gamma g\ker\beta=g'\beta\ker\beta=g'0=0.\]
Since $\gamma$ is monic, we have $g\ker \beta=0$. We also have $f=\ker g$ via Proposition \ref{ext3}, whence we obtain a commutative diagram
\[\begin{tikzcd}[row sep=14 pt]
&&B \arrow[dd,"g"] \\
K_\beta \arrow[urr,bend left,"\ker \beta"] \arrow[r,dashrightarrow,"\exists! h"] \arrow[drr,bend right,"0"'] &A \arrow[ur,"f"] \arrow[dr,"0"'] \\
&&C
\end{tikzcd}\]
Therefore,
\[f'\alpha h=\beta fh=\beta\ker\beta=0.\]
Since $f'$ and $\alpha$ are monic, we have $h=0$. This shows that $\ker \beta=fh=0$, which implies that $\beta$ is monic by Proposition \ref{kereq0}.\\ \\
(ii) Suppose that $\alpha$ and $\gamma$ are epic in $\A$. The reversed commutative diagram of (\ref{seq5}) in the opposite category $\Aop$ is
\[\begin{tikzcd}
0 \arrow[r] & C' \arrow[r, "g'"] \arrow[d, "\gamma"] & B' \arrow[r,"f'"] \arrow[d, "\beta"] & A' \arrow[r] \arrow[d, "\alpha"] & 0 \\
0 \arrow[r] & C \arrow[r, "g"] & B \arrow[r,"f"] & A \arrow[r] & 0
\end{tikzcd}\]
This diagram has exact rows (by virtue of the self-dual property of exactness). And since $\alpha$ and $\gamma$ are monic in $\Aop$, it follows from (i) that $\beta$ is monic in $\Aop$. Hence, $\beta$ is epic in $\A$. \\ \\
(iii) This obviously follows from (i) and (ii), bearing in mind Proposition \ref{isomonoepi}.
\end{proof}

Category theory illuminates the duality present in the short five lemma previously obscured in Theorem \ref{sfl1}. While the proof of Theorem \ref{sfl2} turns out to be more involved than the standard diagram chase, we would argue that the insight gained from this enterprise is worth the cost.

\section{Another proof of the short five lemma}\label{alt}

In this section, we give an alternative proof of the short five lemma forgoing Proposition \ref{ext3} and instead appealing to Definition \ref{ext2} directly. Our proof also presents us with the opportunity of introducing the important concept of pullbacks.

\begin{definition}
Let $\varphi\colon A \to C$ and $\psi\colon B \to C$ be morphisms of a category $\C$. The \textit{pullback} of $\varphi$ and $\psi$ is the triple $(A \times_C B,\pi_A,\pi_B)$, where $A \times_C B \in \Obj(\C)$, $\pi_A\colon A \times_C B \to A$ and $\pi_B\colon A \times_C B \to B$, such that the following square commutes:
\[\begin{tikzcd}
A \times_C B \arrow[r,"\pi_B"] \arrow[d,"\pi_A"'] & B \arrow[d, "\psi"] \\
A \arrow[r,"\varphi"'] & C
\end{tikzcd}\]
and it is terminal with respect to this property, that is, for any commutative square
\[\begin{tikzcd}
D \arrow[r,"g"] \arrow[d,"f"'] & B \arrow[d, "\psi"] \\
A \arrow[r,"\varphi"'] & C
\end{tikzcd}\]
there exists a unique morphism $h\colon D \to A\times_C B$ such that the following diagram commutes:
\[\begin{tikzcd}
D \arrow[rrd,"g",bend left] \arrow[ddr,"f"',bend right] \arrow[dr,dashrightarrow,"\exists!h"]\\
& A \times_C B \arrow[r,"\pi_B"] \arrow[d,"\pi_A"'] & B \arrow[d, "\psi"] \\
& A \arrow[r,"\varphi"'] & C
\end{tikzcd}\]
Dually, let $\varphi'\colon C \to A$ and $\psi'\colon C \to B$ be morphisms of $\C$. The \textit{pushout} of $\varphi'$ and $\psi'$ is the triple $(A \amalg_C B,\mu_A,\mu_B)$, where $A \amalg_C B \in \Obj(\C)$, $\mu_A\colon A \to A \amalg_C B$ and $\mu_B\colon B \to A \amalg_C B$, such that the following square commutes and is initial:
\[\begin{tikzcd}
A \amalg_C B & B \arrow[l, "\mu_B"'] \\
A \arrow[u,"\mu_A"] & C \arrow[l,"\varphi'"] \arrow[u,"\psi'"']
\end{tikzcd}\]
\end{definition}

Note that the previous definition works for arbitrary categories. For abelian categories in particular, we have the following properties:

\begin{proposition}\label{exist}{\normalfont\cite[p. 359]{maclane-homology}}
All pullbacks and pushouts exist in an abelian category $\A$.
\end{proposition}

\begin{proposition}\label{pullback}{\normalfont\cite[p. 578]{aluffi}}
Let
\[\begin{tikzcd}
A \times_C B \arrow[r,"\pi_B"] \arrow[d,"\pi_A"'] & B \arrow[d, "\psi"] \\
A \arrow[r,"\varphi"'] & C
\end{tikzcd}\]
be a pullback in an abelian category. If $\varphi$ is epic, then $\pi_B$ is also epic.
\end{proposition}

For $f\colon A \to B$, we call the kernel of $\cok f$ as the \textit{image} of $f$, which is a morphism $I_f \to B$ denoted as $\img f=\ker(\cok f)$. Dually, the \textit{coimage} of $f$ is $\coim f=\cok(\ker f)$. The reader should verify that in $\Ab$, the inclusion map $\Img f \to B$ and the natural map $A \to \Coim f$ (where $\Coim f=A/\Ker f$) are the image and coimage of $f$, respectively.

\begin{proposition}\label{img}{\normalfont\cite[p. 572]{aluffi}}
Let $f\colon A \to B$ be a morphism in an abelian category. Suppose that $\varphi$ is the unique morphism such that the following diagram commutes:
\begin{equation}\label{im}\begin{tikzcd}[row sep=14 pt]
&&B \arrow[dd,"\cok f"] \\
A \arrow[urr,bend left,"f"] \arrow[r,dashrightarrow,"\exists! \varphi"] \arrow[drr,bend right,"0"'] &I_f \arrow[ur,"\img f"] \arrow[dr,"0"'] \\
&&C_f
\end{tikzcd}\end{equation}
Then, $\varphi$ is epic.
\end{proposition}

We now have every ingredient we need to present our alternative proof of the short five lemma. The proof is a bit complex, but we specifically designed it to involve a lot of universal properties at once. We invite the reader to work through Exercise \ref{sfl3} to reap the potential benefits from this proof. Here is the main diagram for easy reference:

\begin{equation}\label{seq6}\begin{tikzcd}
0 \arrow[r] & A \arrow[r, "f"] \arrow[d, "\alpha"] & B \arrow[r,"g"] \arrow[d, "\beta"] & C \arrow[r] \arrow[d, "\gamma"] & 0 \\
0 \arrow[r] & A' \arrow[r, "f'"] & B' \arrow[r,"g'"] & C' \arrow[r] & 0
\end{tikzcd}\end{equation}

\begin{proof}[Alternative Proof of Theorem \ref{sfl2}{\normalfont(i)}] Suppose that $\alpha$ and $\gamma$ are monic. Since $\gamma g\ker\beta=g'\beta\ker\beta=0$, we have $g\ker \beta=0$. From the universal property of $\ker g$, we see that $\ker \beta$ \textit{factors through} $\ker g$ (i.e., $(\ker g)h=\ker \beta$ for some morphism $h$). By exactness of the first row of (\ref{seq6}), we have $(\cok f)(\ker \beta)=0$. Since we also have $(\cok f)f=0$, we obtain that both $\ker \beta$ and $f$ must factor through $\img f$. Therefore, there exist morphisms $\varphi$ and $\psi$ such that
\begin{gather}
\beta(\img f)\psi=\beta\ker \beta=0\label{eq1}\\
\beta(\img f)\varphi=\beta f=f'\alpha\label{eq2}
\end{gather}
where $\varphi$ is epic by Proposition \ref{img}. Let $P=A \times_{I_f} K_{\beta}$ be the pullback of $\varphi$ and $\psi$:
\begin{equation}\label{square}\begin{tikzcd}
P \arrow[d,"\pi_A"'] \arrow[r,"\pi_{K_\beta}"] & K_\beta \arrow[d,"\psi"] \\
A \arrow[r, "\varphi"'] & I_f
\end{tikzcd}\end{equation}
By multiplying both sides of (\ref{eq1}) and (\ref{eq2}) on the right by $\pi_{K_\beta}$ and $\pi_A$, respectively, we obtain $f'\alpha\pi_A=0$ from the commutativity of (\ref{square}). Since $f'$ and $\alpha$ are monic, we have $\pi_A=0$, which in turn gives us
\[(\ker\beta)\pi_{K_\beta}=(\img f)\psi\pi_{K_\beta}=(\img f)\varphi\pi_A=0.\]
Since $\varphi$ is epic, it follows from Proposition \ref{pullback} that $\pi_{K_\beta}$ is epic. Hence, $\ker\beta=0$, and we are done.
\end{proof}

\section{Exercises}\label{exer}

We compile a problem set concerning diagram chasing and diagrammatic arguments in abelian categories. Some of the problems are taken from the propositions in the main text whose proofs are omitted.

\begin{exercise}
Consider the following commutative diagram of abelian groups (or alternatively, $R$-modules):
\[\begin{tikzcd}
&0 \arrow[d]&0 \arrow[d]&0 \arrow[d] \\
0 \arrow[r] & A_1 \arrow[r, "f_1"] \arrow[d, "\alpha_1"] & B_1 \arrow[r,"g_1"] \arrow[d, "\beta_1"] & C_1 \arrow[r] \arrow[d, "\gamma_1"] & 0 \\
0 \arrow[r] & A_2 \arrow[r, "f_2"] \arrow[d, "\alpha_2"] & B_2 \arrow[r,"g_2"] \arrow[d, "\beta_2"] & C_2 \arrow[r] \arrow[d, "\gamma_2"] & 0 \\
0 \arrow[r] & A_3 \arrow[r, "f_3"] \arrow[d] & B_3 \arrow[r,"g_3"] \arrow[d] & C_3 \arrow[r] \arrow[d] & 0 \\
&0&0&0
\end{tikzcd}\]
Suppose that all three columns of the diagram are exact. Use diagram chasing to prove that if the bottom two rows of the diagram are exact, the top row is also exact. By appealing to the opposite category, prove that the bottom row is exact if the top two rows are exact. This is known as the \textit{nine lemma}.
\end{exercise}

\begin{exercise}
Prove Proposition \ref{kercok}. (Hint: use the fact that $f=\ker \varphi$ for some morphism $\varphi$.)
\end{exercise}

\begin{exercise}
Prove Proposition \ref{isomonoepi}. (Hint: apply Proposition \ref{kercok} and its dual statement to construct a right and left inverse of $f$, respectively.)
\end{exercise}

\begin{exercise}
Prove Proposition \ref{sum}. (Hint: show that there is a natural map $A \amalg B \to A \times B$ that is both monic and epic, and then use Proposition \ref{isomonoepi}.)
\end{exercise}

\begin{exercise}
Show that the conditions $gf=0$ and $(\cok f)(\ker g)=0$ in Definition \ref{ext2} can be replaced by $\img f=\ker g$. (Hint: use Propositions \ref{kercok} and \ref{ext3} to avoid a direct diagrammatic approach.)
\end{exercise}

\begin{exercise}\label{lemma}
Let $f\colon A \to B$ be a morphism in an abelian category, and let $\varphi\colon A \to I_f$ be the unique morphism making the diagram in (\ref{im}) commute. Show that for any morphism $g\colon A \to C$ and monomorphism $h\colon C \to B$ such that $hg=f$, there exists a unique morphism $\psi\colon I_f \to C$ such that the following diagram commutes:
\[\begin{tikzcd}
&C \arrow[dr,tail,"h",bend left]&\\
A \arrow[r,"\varphi"]\arrow[ur,"g",bend left]\arrow[rr,"f"',bend right]&I_f \arrow[r,tail,"\img f"]\arrow[u,dashrightarrow,"\exists!\psi"]&B\\
\end{tikzcd}\]
\end{exercise}

\begin{exercise}
Prove Proposition \ref{exist}. (Hint: given $A,B \in \Obj(\A)$ and their product $(A \oplus B,\rho_A,\rho_B)$, show that if $\omega=\varphi\rho_A-\psi\rho_B$, then $(K_\omega,\rho_A\ker\omega,\rho_B\ker\omega)$ is the pullback of $\varphi\colon A \to C$ and $\psi\colon B \to C$.)
\end{exercise}

\begin{exercise}
Prove Proposition \ref{pullback}. (Hint: first show that $(C,\varphi,\psi)$ is the pushout of $\pi_A$ and $\pi_B$.)
\end{exercise}

\begin{exercise}
Prove Proposition \ref{img}. (Hint: use the result from Exercise \ref{lemma} to show that $\img \varphi$ is epic, and then conclude that $\cok \varphi=\cok (\img \varphi)=0$.)
\end{exercise}

\begin{exercise}\label{sfl3}
Prove Theorem \ref{sfl2}(ii) by mirroring the proof of its dual statement given in Section \ref{alt}. (Hint: formulate and apply the dual statements of Propositions \ref{pullback} and \ref{img}.)
\end{exercise}

\section{Final remarks}

The broad ideas of diagram chasing and diagrammatic arguments have been presented in previous sections. Other than the short five lemma, one of the more important diagram lemmas which has not covered here is the snake lemma, where a diagram chase can be used to construct the aptly named "snaking morphism". For a more thorough exposition on the subject, we suggest the reader to consult \cite{aluffi}.

Of course, many important categorical concepts are unfortunately not included in this article, among which are functors, natural transformations and (co)limits. The universal property of (co)limits is of particular importance due to its generalizing other universal properties we have seen, such as products, pullbacks and kernels. We refer to \cite{leinster} for a closer look at these concepts, and \cite{maclane-categories} for a more advanced reference. Another topic of interest would be a more thorough exposition on homological algebra, of which \cite{maclane-homology} and \cite{rotman} offer a classic and modern treatment, respectively.

\section*{Acknowledgements}

The author would like to thank Hengki Tasman and Wed Giyarti of Universitas Indonesia, both of whom gave helpful supervision for some parts of this article and the author's undergraduate thesis.


\begin{thebibliography}{99}

\bibitem{aluffi} Aluffi, P. (2009). \textit{Algebra: Chapter 0}. Providence: American Mathematical Society.
\bibitem{junhan} Junhan, A. T. (2019). The Freyd-Mitchell Embedding Theorem. \textit{arXiv preprint arXiv:1901.08591}.
\bibitem{lang} Lang, S. (2002). \textit{Algebra} (3rd ed.). New York: Springer.
\bibitem{leinster} Leinster, T. (2014). \textit{Basic category theory}. Cambridge University Press.
\bibitem{maclane-categories} Mac Lane, S. (1978). \textit{Categories for the working mathematician} (2nd ed.). New York: Springer.
\bibitem{maclane-homology} Mac Lane, S. (1994). \textit{Homology}. Berlin: Springer.
\bibitem{rotman} Rotman, J. J. (2008). \textit{An introduction to homological algebra} (2nd ed.). New York: Springer.

\end{thebibliography}
\end{document}